\documentclass[11pt]{article}  
\usepackage[top=2.5cm, bottom=2.5cm, left=2.5cm, right=2.5cm]{geometry}
\usepackage[english]{babel}
\usepackage[T1]{fontenc}
\usepackage[utf8]{inputenc}
\usepackage{amssymb}
\usepackage{lmodern}
\usepackage{microtype}
\usepackage{hyperref}
\usepackage{graphicx}
\usepackage{amsmath}
\usepackage{amsmath,amsfonts}
\usepackage{tikz}
\usepackage{indentfirst}
\usepackage{mathtools}
\usepackage{float}
\usepackage{amsmath,amssymb,amsthm}
\providecommand{\keywords}[1]{\textbf{\textit{Keywords:}} #1}

\newtheorem{theorem}{Theorem}[section]

\newtheorem{lemme}{Lemma}[section]
\newtheorem{prop}{Proposition}[section]

\title{\textbf{\Huge{Directed Plateau Polyhypercubes}}}
\author{Abderrahim Arabi\\
\small USTHB, Faculty of Mathematics\\
\small RECITS Laboratory\\
\small BP 32, El Alia 16111, Bab Ezzouar\\
\small Algiers, Algeria\\
\small\tt rarabi@usthb.dz 
\and
Hacène Belbachir\\
\small USTHB, Faculty of Mathematics\\
\small RECITS Laboratory\\
\small BP 32, El Alia 16111, Bab Ezzouar\\
\small Algiers, Algeria\\
\small\tt hbelbachir@usthb.dz \\ 
\and
Jean-Philippe Dubernard\\
\small University of Rouen-Normandie, Faculty of Science and Technique\\ 
\small LITIS Laboratory\\
\small Avenue de l’université 76800 Saint-Étienne-du-Rouvray\\
\small Rouen, France\\
\small\tt jean-philippe.dubernard@univ-rouen.fr
}
\date{}
\begin{document}
\maketitle
\begin{abstract}
In this paper, we study a particular family of polyhypercubes in dimension  \(d\geq 3\),  the directed plateau polyhypercubes, according to to the width and a new parameter the lateral area. We give an explicit formula and we also propose an expression of the generating function in this case.

\end{abstract}
\keywords{Polyhypercube, polyomino, lateral area, enumeration, generating function.}
\section{Introduction}
In $\mathbb{Z}^2$, a polyomino is a finite union of cells (unit squares), connected
by their edges, without a cut point and defined up to a translation \cite{ref21}. Polyominoes appear in statistical physics, in the phenomenon of percolation \cite{ref12}. The number of polyominoes in general is still an open problem.\\
Polyhypercube are the extension of polyominoes in a dimension $d\geq 3$ \cite{ref4}. In $\mathbb{Z}^d$, a polyhypercube of dimension $d$ is a finite union of cells (unit hypercubes), connected by their hypercubes of dimension $d-1$, and defined up to translation \cite{ref1}. Polyhypercubes are also called $d$-polycubes. They are used in an efficient model for real time validation \cite{ref15} and in representation of finite geometric languages \cite{ref16}. There is no explicit formula for $A_d(n)$ the number of polyhypercubes in dimension $d$ with $n$ cells. Many algorithms were made and values are given for many dimensions (see \cite{ref4}\cite{ref13}).\\
Some families of polyhypercubes were enumerated for instance: the directed plateau, the plateau, the espalier and the pyramid polyhyercubes by the hypervolume and width using Dirichlet convolutions \cite{ref2}. Also, the generating function and asymptotic results were given for the \textit{rs}-directed \cite{ref3}, polyhypercubes that can be split into directed strata. Also for $d=3$, called polycubes,  the directed plateau and the plateau polycubes were enumerated according to the lateral area \cite{ABD}.\\
In this paper, we introduce a new parameter: the lateral area of a polyhypercube for a dimension $d\geq 3$. Using this parameter and the width, we enumerate the family of directed plateau polyhypercubes.

\section{Preliminaries}
Let $(0,\vec{i},\vec{j})$ be an orthonormal coordinate.  
The area of a polyomino is the number of its cells, its width is the number of its columns and its height is the number of its lines. A polyomino is column-convex if its intersection with any vertical line is connected. A North (resp. East) step is a movement of one unit in $\vec{i}$-direction (resp. $\vec{j}$-direction).
A polyomino is directed if from a distinguished cell of the polyomino called root, we reach any other cell by a path that uses only North or East steps.   

Let $(0, \vec{i_1}, \vec{i_2},..., \vec{i_d} )$ be an orthonormal coordinate system. The volume of polyhypercube is the number of its hypercubes. The width is the difference between its greatest index and its smallest index according to $\vec{i_1}$. 
We define the lateral area of an polyhypercube as the sum of the areas of the polyominoes obtained by its projection on the planes $(\vec{i_1}, \vec{i_l})$ with $2\leq l \leq d$. An elementary step is a positive move of one unit along the axis $\vec{i_j}$ with $1\leq j \leq d$. A polyhypercube is directed, if each cell can be reached from a distinguished cell called root, by a path only made by elementary steps.
An stratum is polyhypercube of width one. A plateau is an hyperrectangular stratum.
A directed plateau polyhypercube is polyhypercube whose strata are plateaus.\\
To avoid many steps of calculation we use the following useful convention for binomial coefficient, for $n\geq 0$,
$$
\binom{n}{k}=0 \: \:\text{for}\: \:k<0 \:\:\text{or}\:\: k>n.
$$
    
\section{Explicit enumeration of directed plateau polyhypercube} 
To enumerate directed plateau polyhypercubes, we characterize their projections.
\begin{theorem}\label{thproj}
For $k\geq 1$, $d\geq 3$ and $2\leq l \leq d$, the projection of a directed plateau polyhypercube of width $k$ on a plane $(\vec{i_1}, \vec{i_l})$ gives a directed column-convex polyomino of width $k$.
\end{theorem}

\begin{proof}
We have the hypothesis that for each directed plateau polyhypercube we associate a $(d-1)$-tuple of polyominoes obtained by the projection of the polyhypercube on the planes $(\vec{i_1}, \vec{i_l})$ for $2\leq l\leq d$.\\
Let $A$ be a $(d-1)$-tuple of polyominoes and suppose that we can build two different polyhypercubes. It means that the polyhypercubes are different in at least one plateau
on $(\vec{i_1}, \vec{i_l})$, with $2\leq l \leq d$.\\
This implies that the two polycubes have different coordinates in $(\vec{i_1}, \vec{i_l})$ and their projections on these planes are different, it contradicts the initial hypothesis. 
\end{proof}
In order to enumerate the directed plateau polyhypercubes, we use the following lemma.
\begin{lemme}[\cite{ref9}]
\label{LL}
Let $c_{k,n}$ be the number of directed column-convex polyominoes having $k$ columns and area $n$. Then for $k\geq 1$ and $n\geq k$,
$$c_{k,n}=\binom{n+k-2}{n-k}.$$
\end{lemme} 
\begin{theorem}
\label{T1}
Let $p_{d,k,n}$ be the number of directed plateau polyhypercubes of dimension $d$, of width $k$ and having a lateral area $n$. Then for $d\geq 3$ and $n \geq (d-1)k$,
$$p_{d,k,n}=\sum_{j_2+j_3+...+j_d=n} \prod_{l=2}^{d}\binom{j_l+k-2}{j_l-k}.$$

%%$$\sum_{\substack{j_2,j_3,...,j_{d}\geq k \\j_2+j3+...+j_d=n}}\binom{j_2+k-2}{j_2-k}\binom{j_3+k-2}{j_3-k}...\binom{j_d+k-2}{j_d-k}.$$

%%$$p_{d,k,n}=\sum_{{j_2,...,j_{d-1}\geq k}}\binom{j_2+k-2}{j_2-k}\binom{j_3+k-2}{j_3-k}...\binom{j_{d-1}+k-2}{j_{d-1}-k}\binom{m_1+k-2}{m_1-k}$$
%%where $m=n-j_2-j_3-...-j_{d-1}$.
\end{theorem}
\begin{proof}
If a polyhypercube of dimension $d$ has a width $k$ and a lateral area $n$, then from Theorem \ref{thproj} each of its projection on a plane $(\vec{i_1}, \vec{i_l})$ gives a polyomino of width $k$ and area $j_l$, with $j_l\geq k$, for $l$ such that $2 \leq l \leq d$. And the sum of the areas of all polyominoes obtained from projections is equal to $n$.\\
From Lemma \ref{LL}, it is known that the number of column-convex 
polyominoes having $k$ columns and area $j_l$ is equal to $\binom{j_l+k-2}{j_l-k}$.
Therefore the number of directed plateau polyhypercubes of dimension $d$, width $k$ and whose projections on the planes $(\vec{i_1},\vec{i_l})$ give a polyomino of area $j_l$ is
$\prod_{l=2}^{d}\binom{j_l+k-2}{j_l-k}$, with $2 \leq l \leq d$. Therefore, the formula is obtained by summing for all values of $j_l$, $2\leq l \leq d$.
%%The sum of the areas of the polyominoes $\sum_{l=2}^{d}j_l=n$, thus $j_{d}=n-j_2-j_3...-j_{d-1}$. \\
%%By replacing $j_{d-1}$ by its value, we obtained the formula. 
\end{proof}
\begin{lemme}[\cite{ref23}]
\label{L}
For $x$, $y$, $n$ and $k$ integers,
$$\sum_{k=0}^{n}\binom{x+k}{k}\binom{y+n-k}{n-k}=\binom{x+y+n+1}{n}.\\$$
\end{lemme}
\noindent
For more properties on Vandermonde's convolutions see \cite{Bel15}. 
\begin{theorem}\label{T2}
For $d\geq 3$ and $n \geq (d-1)k$,
$$p_{d,k,n}=\binom{n+(d-1)k-d}{n-(d-1)k}.$$
\end{theorem}

\begin{proof}
Let us prove the result by induction. In dimension $3$, according to \cite{ABD}, the number of directed plateau polycubes of width $k$ and having a lateral area equal to $n$ is equal to,
$$\binom{n+2k-3}{n-2k}.$$
Here, this result corresponds to the case of $d=3$.
Let us now suppose that, for a given $d$
$$p_{d,k,n}=\binom{n+(d-1)k-d}{n-(d-1)k},$$
and let us prove that 
$$p_{d+1,k,n}=\binom{n+dk-d-1}{n-dk}.$$
From Theorem \ref{T1}
\begin{align*}
p_{d+1,k,n}&=\sum_{j_2+j_3+...+j_d=n}\prod_{l=2}^{d+1}\binom{j_l+k-2}{j_l-k} \\ \\
&=\sum_{j_{d+1}=k}^{n-(d-1)k}\binom{j_{d+1}+k-2}{j_{d+1}-k}\sum_{j_2+j_3+...+j_{d}=n-j_{d+1}}\prod_{l=2}^{d}\binom{j_l+k-2}{j_l-k}\\ \\
& =\sum_{j_{d+1}=k}^{n-(d-1)k}\binom{j_{d+1}+k-2}{j_{d+1}-k}\binom{n-j_{d+1}+(d-1)k-d}{n-j_{d+1}-(d-1)k}.
\end{align*}

Setting $i=j_{d+1}-k$, $m=n-dk$, $a=2k-2$ and $b=2(d-1)k-d$, we obtain
$$\sum_{i=0}^{m}\binom{i+a}{i}\binom{m-i+b}{m-i}.$$
Using Lemma \ref{L} and replacing $m$, $a$ and $b$ by their values we get the formula.

\end{proof}
\section{Generating functions}
Let $P_{d,k}(t)$ be the generating function of the directed plateaus polyhypercubes of dimension $d$ and width $k$ according to the lateral area.
$$P_{d,k}(t)=\sum_{n\geq 1} p_{d,k,n}t^n.$$ 
\noindent

\begin{prop}\label{pr}
For $k\geq 1$, we have
$$P_{d,k}(t)=\frac{t^{k(d-1)}}{(1-t)^{2k(d-1)-(d-1)}}. $$
\end{prop}

\begin{proof}
Using Theorem\ref{T1}, we get
$$P_{d,k}(t)=\sum_{n\geq 0}\sum_{j_2+j_3+...+j_d=n}\prod_{l=2}^{d}\binom{j_l+k-2}{j_l-k}t^n.$$
For $l$ such that $2\leq l \leq d$, if $j_l < k$ or $j_l > n-(d-1)k$ then $\prod_{l=2}^{d}\binom{j_l+k-2}{j_l-k}=0$. Therefore,

\begin{align*}
P_{d,k}(t)&=\sum_{n\geq 0}\sum_{j_2+j_3+...+j_d=n}\prod_{l=2}^{d}\binom{j_l+k-2}{j_l-k}t^n\\
&=\bigg(\sum_{j_2\geq 0}\binom{j_2+k-2}{j_2-k}t^{j_2}\bigg)\bigg(\sum_{j_3\geq 0}\binom{j_3+k-2}{j_3-k}t^{j_3}\bigg)\cdots \bigg(\sum_{j_d\geq 0}\binom{j_d+k-2}{j_d-k}t^{j_d}\bigg).
\end{align*}
It is know from Barcucci \textit{et al.} \cite{ba93}, that
$$\sum_{n\geq 0}\binom{n+k-2}{n-k}t^n=\frac{t^k}{(1-t)^{2k-1}},$$
thus we get the result. 
\end{proof}
\noindent
Let $$P_d(t,x):=\sum_{k\geq 1}P_{d,k}(t)x^k,$$ be the generating function of directed plateau polyhypercubes according to the width (coded by $x$) and the lateral area (coded by $t$). \\
From Proposition \ref{pr}, then
\begin{equation}\label{a}
P_d(t,x)=\frac{x t^{d-1}(1-t)^{d-1}}{(1-t)^{2(d-1)}-xt^{d-1}}.
\end{equation}

From this expression we deduce the following theorem.
\begin{theorem}
Let $P_d(t)$ be the generating function of directed plateau polyhypercubes according to the lateral area. Then for $d\geq 3$,
$$P_d(t)=\frac{t^{d-1}(1-t)^{d-1}}{(1-t)^{2(d-1)}-t^{d-1}}.$$
\end{theorem}
\begin{proof}
We set $x=1$ in equation \ref{a}.
\end{proof}


\begin{thebibliography}{1}
\begin{small}
\bibitem{ref4}
G.~Aleksandrowicz and G.~Barequet. 
\newblock Counting d-dimensional polycubes and nonrectangular planar polyominoes. 
\newblock {\em Int. J. Comput. Geometry Appl.}, 19(3):215–229, 2009.

\bibitem{ABD}
A. Arabi, H. Belbachir and J.-Ph. Dubernanrd. 
\newblock Plateau Polycubes and Lateral Area. 
\newblock arXiv:1811.05707[math.CO], 14 Nov 2018.

\bibitem{ba93}
E. Barcucci, R. Pinzani and R. Sprugnoli
\newblock Directed column-convex polyominoes by recurrence relations.
\newblock{\em Lecture Notes in Computer Science}, 668:282-298, 1993.

\bibitem{Bel15}
H. Belbachir.
\newblock A combinatorial contribution to the multinomial Chu-Vandermonde convolution.
\newblock {\em Les Annales RECITS}, Vol. 01:27-32, 2014.

\bibitem{ref21}
M. Bousquet-Mélou.
\newblock A method for the enumeration of various classes of column-convex polygons. \newblock{\em Discrete Mathematics}, 154(1-3):1–25, 1996.

\bibitem{ref2}
C.~Carrée, N.~Debroux, M.~Deneufchâtel, J.-Ph.~Dubernard, C.~Hillairet, J.~G.~Luque and O.~Mallet. 
\newblock Enumeration of polycubes and Dirichlet convolutions.
\newblock {\em Journal of Integer Sequences}, vol 18, 2015.

\bibitem{ref1}
J.~M.~Champarnaud, Q.~Cohen-Solal, J.-Ph. Dubernard and H.~Jeanne.
\newblock Enumeration of specific classes of polycubes.
\newblock {\em The Electronic Journal of Combinatorics}, 20(4), 2013.

\bibitem{ref16}
J.-M. Champarnaud, J.-Ph. Dubernard, and H. Jeanne. \newblock An efficient algorithm to test whether a binary and prolongeable regular language is geometrical. \newblock{\em Int. J. Found. Comput. Sci.}, 20(4) :763–774, 2009.

\bibitem{ref3}
J.-M.~Champarnaud, J.-Ph.~Dubernard and H.~Jeanne.
\newblock A generic method for the enumeration of various classes of directed
  polycubes.
\newblock {\em Discrete Mathematics and Theoretical Computer Science}, 15(1),
  2013.

\bibitem{ref9}
M.~P.~Delest and S.~Dulucq. 
\newblock Enumeration of directed column-convex animals with given perimeters and area. \newblock {\em In Croatica Chimica Acta}, volume 66, pages 59–80, 1993.

\bibitem{ref23}
G. W. Gould.
\newblock Combinatorial, A Standardized Set of Tables Listing 500 Binomial Coefficient Summations. Revised Edition.
\newblock{\em Morgantown, W. Va.}, 1972. 


\bibitem{ref15}
G. Largeteau and D. Geniet. \newblock Quantification du taux d'invalidité d'applications temps-réel à contraintes strictes. \newblock{\em Technique et Science Informatiques}, 27(5) :589–625, 2008.


\bibitem{ref13}
S. Luther and S. Mertens. \newblock Counting lattice animals in high dimension. 
\newblock{\em J. of Statistical Mechanics: Theory and Experiment}, 9 :546-565, 2011.


\bibitem{ref12}
H. N. V. Temperley.
\newblock Combinatorial problems suggested by the statistical mechanics of domains and of rubber-like molecules. 
\newblock{\em Phys. Rev}, 103 :1-16, 1956. 


\end{small}
\end{thebibliography}
\end{document}